\documentclass[oneside,english]{amsart}
\usepackage[T1]{fontenc}
\usepackage[latin9]{inputenc}
\usepackage{verbatim}
\usepackage{amsthm}
\usepackage{setspace}
\usepackage{amssymb}
\usepackage{esint}
\makeatletter
  \theoremstyle{plain}
  \newtheorem{thm}{Theorem}
  \theoremstyle{definition}
  \newtheorem{defn}{Definition}
  \theoremstyle{plain}
  \newtheorem{lem}{Lemma}
 \ifx\proof\undefined\
   \newenvironment{proof}[1][\proofname]{\par
     \normalfont\topsep6\p@\@plus6\p@\relax
     \trivlist
     \itemindent\parindent
     \item[\hskip\labelsep
           \scshape
       #1]\ignorespaces
   }{%
     \endtrivlist\@endpefalse
   }
   \providecommand{\proofname}{Proof}
 \fi
  \theoremstyle{plain}
  \newtheorem{cor}{Corollary}
  \theoremstyle{remark}
  \newtheorem{rem}{Remark}
  \theoremstyle{definition}
  \newtheorem{example}{Example}

\usepackage[alphabetic]{amsrefs}

\makeatother

\usepackage{babel}

\begin{document}
\title[Negligible Variation and the Change of Variables Theorem]{Negligible Variation and\\the Change of Variables Theorem}

\author{Vermont Rutherfoord and Yoram Sagher}
\begin{abstract}
In this note we prove a necessary and sufficient condition for the
change of variables formula for the HK integral, with implications
for the change of variables formula for the Lebesgue integral. As
a corollary, we obtain a necessary and sufficient condition for the
Fundamental Theorem of Calculus to hold for the HK integral.
\end{abstract}

\address{Vermont Rutherfoord, Yoram Sagher\\
Florida Atlantic University\\
Department of Mathematical Sciences\\
777 Glades Road\\
Boca Raton, FL 33431}

\email{vermont.rutherfoord@gmail.com\\
yoram.sagher@gmail.com}


\maketitle
J. Serrin and D.E. Varberg \cite{key-2} proved the following change
of variables theorem for the Lebesgue integral.
\begin{thm}
Assume that $g:\left[a,b\right]\rightarrow\mathbb{R}$ is differentiable%
\footnote{Unless otherwise noted, we use differentiable to mean finitely differentiable.%
} almost everywhere and that $f:\mathbb{R}\rightarrow\mathbb{R}$ is
Lebesgue integrable on $\left[c,d\right]\supseteq g\left(\left[a,b\right]\right)$.
Then $\left(f\circ g\right)\cdot g'$ is Lebesgue integrable on $\left[a,b\right]$
and the change of variables formula\begin{equation}
\int_{g\left(\alpha\right)}^{g\left(\beta\right)}f\left(u\right)du=\int_{\alpha}^{\beta}f\left(g\left(s\right)\right)g'\left(s\right)ds\label{eq:cov1}\end{equation}
holds for all $\alpha$, $\beta$ in $\left[a,b\right]$ if and only
if $F\circ g$ is absolutely continuous, where $F\left(x\right):=\int_{g\left(a\right)}^{x}f\left(u\right)du$.
\end{thm}
K. Krzyzewski \cite{key-5} and G. Goodman \cite{key-7} proved several
sufficient but not necessary conditions for \eqref{eq:cov1} to hold
for the Denjoy and Perron integrals. Both integrals are equivalent
to Henstock-Kurzweil (HK) integral, for which we establish necessary
and sufficient conditions for the change of variables theorem. As
a consequence, we obtain the optimal condition for which \eqref{eq:cov1}
holds for a fixed $\alpha,\beta$ without the requirement that it
holds for every subinterval. Furthermore, we show that even when $\int_{g\left(\alpha\right)}^{g\left(\beta\right)}f\left(u\right)du$
is a Lebesgue integral,  \eqref{eq:cov1} holds under weaker conditions
than those of Theorem 1.

\section*{Preliminaries}

For excellent presentations of the HK integral, see \cite{key-1},
\cite{key-4}, and \cite{key-6}. For the reader's convenience, we
include the basic definitions necessary to follow the exposition below.

We denote the closed interval $\left[a_{j},b_{j}\right]$ as $I_{j}$
and $\left|I_{j}\right|=b_{j}-a_{j}$.

\begin{defn}
A \emph{tagged partition $P$} \emph{of $\left[a,b\right]$}, denoted
$P\left[a,b\right]$, is a finite set of the form $\left\{ \left(x_{j},I_{j}\right):1\le j\le n\right\} $
such that $\bigcup_{j=1}^{n}I_{j}=\left[a,b\right]$, $x_{j}\in I_{j}$
for all $j$, and $i\ne j$ implies that $\left(a_{i},b_{i}\right)\cap\left(a_{j},b_{j}\right)=\emptyset$.
\end{defn}
\begin{defn}
A \emph{gauge} for the interval $\left[a,b\right]$ is a function
from $\left[a,b\right]$ to the positive real numbers, $\mathbb{R}^{+}$.
\end{defn}
\begin{defn}
A tagged partition $P\left[a,b\right]=\left\{ \left(x_{j},I_{j}\right):1\le j\le n\right\} $
is said to be \emph{subordinate} to the gauge $\delta$ if $I_{j}\subseteq\left(x_{j}-\delta\left(x_{j}\right),x_{j}+\delta\left(x_{j}\right)\right)$
for every $j$.
\end{defn}
\begin{defn}
Given $f:\left[a,b\right]\rightarrow\mathbb{R}$ and a tagged partition
$P\left[a,b\right]=\left\{ \left(x_{j},I_{j}\right):1\le j\le n\right\} $,
$R_{P}f:=\sum_{j=1}^{n}f\left(x_{j}\right)\cdot\left|I_{j}\right|$
is called their \emph{Riemann sum}.
\end{defn}
\begin{defn}
A function $f:\left[a,b\right]\rightarrow\mathbb{R}$ is said to be
\emph{HK integrable} over $\left[a,b\right]$ if there exists a number,
$(HK)\int_{a}^{b}f\left(x\right)dx$, so that for any $\epsilon>0$
there exists a gauge $\delta$ on $\left[a,b\right]$ such that any
tagged partition $P\left[a,b\right]$ subordinate to $\delta$ satisfies
$\left|R_{P}f-(HK)\int_{a}^{b}f\left(x\right)dx\right|<\epsilon$.
We also define $(HK)\int_{b}^{a}f\left(x\right)dx=-(HK)\int_{a}^{b}f\left(x\right)dx$.
\end{defn}

Since two integrands which differ only on a set of measure zero have
the same HK integral \cite[Theorem 2.5.6]{key-4}, we adopt the convention
that, given a function $f$ that is defined almost everywhere, its
HK integral is the integral of the function that is equal to $f$
where $f$ is defined, and is $0$ where $f$ is not defined.

\begin{defn}
A function $f$ has \emph{negligible variation} on a set $E\subseteq\left[a,b\right]$
if for any $\epsilon>0$ there exists a gauge $\delta$ on $\left[a,b\right]$
such that for any tagged partition $P\left[a,b\right]=\left\{ \left(x_{j},I_{j}\right):1\leq j\leq n\right\} $
subordinate to $\delta$, $\sum_{x_{j}\in E}\left|f\left(b_{j}\right)-f\left(a_{j}\right)\right|<\epsilon$.
\end{defn}
From this point on, we will denote $\Delta_{j}f=f\left(b_{j}\right)-f\left(a_{j}\right)$.

The definition of negligible variation and a broad range of applications
was introduced by \cite{key-9}. One of them is the following theorem,
which was proven in the following necessary and sufficient form in
Theorem 5.12 of \cite{key-1}.
\begin{thm}
[Fundamental Theorem of Calculus for the HK Integral]\label{thm:bartle}\textup{~}\\
\textup{$F\left(x\right)-F\left(a\right)=(HK)\int_{a}^{x}f\left(s\right)ds$}
for all $x\in\left[a,b\right]$ if and only if there exists a set
$E\subseteq\left[a,b\right]$ such that $F'\left(x\right)=f\left(x\right)$
for all $x\in E$ and $\left[a,b\right]\backslash E$ is a set of
measure zero on which $F$ has negligible variation.
\end{thm}
In the context of the HK integral, negligible variation plays a role
analogous to that played by absolute continuity in the context of
the Lebesgue integral. Corollary~14.8 of \cite{key-1} proved the
following relationship between the two.
\begin{thm}
$F$ is absolutely continuous on $\left[a,b\right]$ if and only if
$F$ is of bounded variation on $\left[a,b\right]$ and $F$ has negligible
variation on every subset of $\left[a,b\right]$ that has measure
zero.
\end{thm}

\section{Change of Variables on a Single Interval}
\begin{defn}
A function $f$ has \emph{negligible conditional variation} on a set
$E\subseteq\left[a,b\right]$ if for any $\epsilon>0$ there exists
a gauge $\delta$ on $\left[a,b\right]$ such that for any tagged
partition $P\left[a,b\right]=\left\{ \left(x_{j},I_{j}\right):1\leq j\leq n\right\} $
subordinate to $\delta$, $\left|\sum_{x_{j}\in E}\Delta_{j}f\right|<\epsilon$.
\end{defn}
Some functions may have negligible conditional variation but not negligible
variation on the set of points where they fail to be differentiable;
a simple example is the indicator function of an open interval contained
in $\left[a,b\right]$. We will present a continuous function with
this property in Example~\ref{exa:strictnegvar} of Section~\ref{sec:Negligible-Variation}.

The same examples show that, although a function that has negligible
variation on a set also has negligible variation on all its subsets,
this is not true for negligible conditional variation.

We will need the following theorem, which was proven in \cite{key-5}
and \cite{key-2}. For a stronger version of the theorem, see Theorem
\ref{thm:negvarderiv}.
\begin{thm}
\label{thm:svzero}If $g$ has a derivative (finite or infinite) on
a set $E$ and $g\left(E\right)$ has measure zero, then $g'=0$ almost
everywhere on $E$. \end{thm}
\begin{lem}
\label{zeroset}Assume that both $g:\left[a,b\right]\rightarrow D$
and $F:D\rightarrow\mathbb{R}$ have derivatives almost everywhere
and that $f=F'$ almost everywhere. Then $g'\left(x\right)=0$ at
almost every $x\in\left[a,b\right]$ where the equality\begin{equation}
\left(F\circ g\right)'\left(x\right)=\left(f\circ g\cdot g'\right)\left(x\right)\label{eq:chainrule}\end{equation}
 fails, that is to say where \eqref{eq:chainrule} is false or either
side is undefined.\end{lem}
\begin{proof}
Let $Z$ be the null set where $F$ does not have a derivative equal
to $f$. By Theorem \ref{thm:svzero}, $g'\left(x\right)=0$ for almost
every $x\in g^{-1}\left(Z\right)$. On the complement of $g^{-1}\left(Z\right)$,
\eqref{eq:chainrule} holds at all $x$ where $g'\left(x\right)$
exists, and so almost everywhere.\end{proof}
\begin{thm}
\label{thm:subs1}Assume that $g:\left[a,b\right]\rightarrow\mathbb{R}$
has a derivative almost everywhere and that $f:\mathbb{R}\rightarrow\mathbb{R}$
is HK integrable on every interval with endpoints in the range of
$g$.%
\footnote{If the HK integral of $f$ exists on an interval then it also exists
on every subinterval \cite[Corollary 3.8]{key-1}. It is therefore
sufficient to require that $f$ be HK integrable over an interval
containing the range of $g$.%
} Define $F\left(x\right):=(HK)\int_{g\left(a\right)}^{x}f\left(u\right)du$.
Then $\left(f\circ g\right)\cdot g'$ is HK integrable on $\left[a,b\right]$
and the change of variables formula \[
(HK)\int_{g\left(a\right)}^{g\left(b\right)}f\left(u\right)du=(HK)\int_{a}^{b}f\left(g\left(s\right)\right)g'\left(s\right)ds\]
holds if and only if $F\circ g$ has negligible conditional variation
on the set where $\left(F\circ g\right)'=f\circ g\cdot g'$ fails.\end{thm}
\begin{proof}
Let $B$ be the set where $\left(F\circ g\right)'=f\circ g\cdot g'$
fails and assume $F\circ g$ has negligible conditional variation
there. Let $h\left(x\right)=0$ if $x\in B$ and $h\left(x\right)=g'\left(x\right)$
otherwise.

By Theorem~\ref{thm:bartle}, $F$ has a derivative equal to $f$
almost everywhere and so by Lemma~\ref{zeroset}, $g'=h=0$ almost
everywhere on $B$. Consequently, \[
(HK)\int_{a}^{b}\left(f\circ g\cdot g'\right)\left(s\right)ds=(HK)\int_{a}^{b}\left(f\circ g\cdot h\right)\left(s\right)ds.\]
 Since $F\circ g$ has a derivative on the complement of $B$, there
exists for any $\epsilon>0$ a function $\eta_{\epsilon}:\left[a,b\right]\backslash B\rightarrow\mathbb{R}^{+}$
such that if $y\in\left[x-\eta_{\epsilon}\left(x\right),x+\eta_{\epsilon}\left(x\right)\right]\cap\left[a,b\right]$
then \[
\left|\left(F\circ g\right)'\left(x\right)\cdot\left(y-x\right)-\left(\left(F\circ g\right)\left(y\right)-\left(F\circ g\right)\left(x\right)\right)\right|<\epsilon\left|y-x\right|/\left(b-a\right)\]
Also, because $F\circ g$ has negligible conditional variation on
$B$, there exists a gauge $\delta_{1}$ on $\left[a,b\right]$ such
that for any tagged partition $P\left[a,b\right]=\left\{ \left(x_{j},I_{j}\right):1\leq j\leq n\right\} $
subordinate to $\delta_{1}$, $\left|\sum_{x_{j}\in B}\Delta_{j}f\right|<\epsilon/2$.
Let $\delta$ be a gauge on $\left[a,b\right]$ so that $\delta\left(x\right)=\eta_{\epsilon/2}\left(x\right)$
if $x\notin B$ and $\delta\left(x\right)=\delta_{1}\left(x\right)$
if $x\in B$.

Consider a tagged partition $P\left[a,b\right]=\left\{ \left(x_{j},I_{j}\right):1\leq j\leq n\right\} $
subordinate to $\delta$. The Riemann sum of $f\circ g\cdot h$ corresponding
to this tagged partition is \begin{eqnarray*}
R_{P}f\circ g\cdot h & = & \overbrace{\sum_{\substack{x_{j}\in B}
}\left(f\circ g\cdot h\right)\left(x_{j}\right)\cdot\left|I_{j}\right|}^{=0}+\sum_{\substack{x_{j}\notin B}
}\left(f\circ g\cdot h\right)\left(x_{j}\right)\cdot\left|I_{j}\right|\\
 & = & \sum_{\substack{x_{j}\notin B}
}\left(F\circ g\right)'\left(x_{j}\right)\cdot\left|I_{j}\right|\\
 & = & \sum_{\substack{x_{j}\notin B}
}\left[\left(F\circ g\right)'\left(x_{j}\right)\cdot\left|I_{j}\right|-\Delta_{j}\left(F\circ g\right)\right]\\
 &  & +\sum_{\substack{x_{j}\notin B}
}\Delta_{j}\left(F\circ g\right)\\
 & = & \sum_{\substack{x_{j}\notin B}
}\left[\left(F\circ g\right)'\left(x_{j}\right)\cdot\left|I_{j}\right|-\Delta_{j}\left(F\circ g\right)\right]\\
 &  & +\sum_{\substack{x_{j}\in\left[a,b\right]}
}\Delta_{j}\left(F\circ g\right)-\sum_{\substack{x_{j}\in B}
}\Delta_{j}\left(F\circ g\right)\\
 & = & \sum_{\substack{x_{j}\notin B}
}\left[\left(F\circ g\right)'\left(x_{j}\right)\cdot\left|I_{j}\right|-\Delta_{j}\left(F\circ g\right)\right]\\
 &  & +(HK)\int_{g\left(a\right)}^{g\left(b\right)}f\left(u\right)du-\sum_{\substack{x_{j}\in B}
}\Delta_{j}\left(F\circ g\right).\end{eqnarray*}
 And so \begin{eqnarray}
 &  & R_{P}f\circ g\cdot h-(HK)\int_{g\left(a\right)}^{g\left(b\right)}f\left(u\right)du\label{eq:main}\\
 & = & \sum_{\substack{x_{j}\notin B}
}\left[\left(F\circ g\right)'\left(x_{j}\right)\cdot\left|I_{j}\right|-\Delta_{j}\left(F\circ g\right)\right]-\sum_{\substack{x_{j}\in B}
}\Delta_{j}\left(F\circ g\right).\nonumber \end{eqnarray}
 Since $F\circ g$ has negligible conditional variation on $B$ and
$\delta$ is chosen accordingly, \[
\left|\sum_{\substack{x_{j}\in B}
}\Delta_{j}\left(F\circ g\right)\right|<\epsilon/2.\]
 Also, for any $x_{j}\notin B$, \[
\left|\left(F\circ g\right)'\left(x_{j}\right)\cdot\left|I_{j}\right|-\Delta_{j}\left(F\circ g\right)\right|<\epsilon\left|I_{j}\right|/2\left(b-a\right),\]
 and so \[
\left|\sum_{\substack{x_{j}\notin B}
}\left(F\circ g\right)'\left(x_{j}\right)\cdot\left|I_{j}\right|-\Delta_{j}\left(F\circ g\right)\right|<\epsilon/2.\]
 Therefore\[
\left|R_{P}f\circ g\cdot h-(HK)\int_{g\left(a\right)}^{g\left(b\right)}f\left(u\right)du\right|<\epsilon,\]
 proving that $(HK)\int_{a}^{b}f\left(g\left(s\right)\right)g'\left(s\right)ds$
exists and is equal to $(HK)\int_{g\left(a\right)}^{g\left(b\right)}f\left(u\right)du$.

Conversely, choose $\epsilon>0$, let $B$, $h$, and $\eta_{\epsilon}$
be defined as above, and assume $(HK)\int_{g\left(a\right)}^{g\left(b\right)}f\left(u\right)du=(HK)\int_{a}^{b}f\left(g\left(s\right)\right)h\left(s\right)ds$.
Thus there exists a gauge $\delta_{1}$ on $\left[a,b\right]$ so
that for any tagged partition $P$ subordinate to $\delta_{1}$, \begin{equation}
\left|R_{P}f\circ g\cdot h-(HK)\int_{g\left(a\right)}^{g\left(b\right)}f\left(u\right)du\right|<\epsilon/2.\label{eq:hkint1}\end{equation}
 Let $\delta$ be a gauge on $\left[a,b\right]$ so that $\delta\left(x\right)=\min\left\{ \delta_{1}\left(x\right),\eta_{\epsilon/2}\left(x\right)\right\} $
if $x\notin B$, and $\delta\left(x\right)=\delta_{1}\left(x\right)$
if $x\in B$.

Choose any tagged partition $P\left[a,b\right]=\left\{ \left(x_{j},\left[a_{j},b_{j}\right]\right):1\leq j\leq n\right\} $
subordinate to $\delta$. Consequently it is also subordinate to $\delta_{1}$,
and so \eqref{eq:hkint1} holds. By \eqref{eq:main}, \begin{eqnarray*}
\left|\sum_{\substack{x_{j}\notin B}
}\left(F\circ g\right)'\left(x_{j}\right)\cdot\left|I_{j}\right|-\Delta_{j}\left(F\circ g\right)-\sum_{\substack{x_{j}\in B}
}\Delta_{j}\left(F\circ g\right)\right| & < & \epsilon/2.\end{eqnarray*}
 Also, \[
\left|\sum_{\substack{x_{j}\notin B}
}\left(F\circ g\right)'\left(x_{j}\right)\cdot\left|I_{j}\right|-\Delta_{j}\left(F\circ g\right)\right|<\epsilon/2.\]
 Therefore \[
\left|\sum_{\substack{x_{j}\in B}
}\Delta_{j}\left(F\circ g\right)\right|<\epsilon,\]
 proving that $F\circ g$ has negligible conditional variation on
$B$.
\end{proof}
By taking $F\left(x\right)=x$, we obtain as a corollary the following
necessary and sufficient condition for the Fundamental Theorem of
Calculus for the HK integral to hold for a particular interval, rather
than all subintervals.
\begin{cor}
\label{cor:fund}Assume that $g:\left[a,b\right]\rightarrow\mathbb{R}$
is differentiable almost everywhere on $\left[a,b\right]$. Then $g'$
is HK integrable on $\left[a,b\right]$ and $g\left(b\right)-g\left(a\right)=(HK)\int_{a}^{b}g'\left(s\right)ds$
if and only if $g$ has negligible conditional variation on the set
where it is not differentiable.
\end{cor}
Theorem \ref{thm:subs1} complements Theorem 1 by obtaining a necessary
and sufficient condition for change of variables to hold on a single
interval. However, even when one side of \eqref{eq:cov1} is a Lebesgue
integral, the integral on the other side sometimes must be taken in
the HK sense. For example, take $g$ as any function that is not an
indefinite Lebesgue integral but which satisfies the condition of
Corollary \ref{cor:fund} and let $f\left(x\right)=1$.

\section{\label{sec:allsubint}Change of Variables on All Subintervals}

The following recasting of the Saks-Henstock Lemma for negligible
variation provides a corollary to Theorem \ref{thm:subs1} where change
of variables holds for each subinterval and, in this sense, provides
the precise HK analog of the Serrin and Varberg theorem.
\begin{lem}
\label{lem:saks}Let $f$ be a real-valued function on $\left[a,b\right]$
and $E$ a subset of $\left[a,b\right]$. Assume that $f$ has negligible
conditional variation on $E\cap\left[\alpha,\beta\right]$ for every
$\left[\alpha,\beta\right]\subseteq\left[a,b\right]$. Then $f$ has
negligible variation on $E$.\end{lem}
\begin{proof}
Choose $\epsilon>0$ and let $\delta$ be a gauge on $\left[a,b\right]$
such that for any tagged partition $P\left[a,b\right]=\left\{ \left(x_{j},I_{j}\right):1\leq j\leq n\right\} $
subordinate to $\delta$, \[
\left|\sum_{x_{j}\in E}\Delta_{j}f\right|<\epsilon\]

Fix such a tagged partition $P$ and choose $\epsilon'>0$. Since
$f$ has negligible conditional variation on each $I_{j}$, there
exists a gauge, $\delta_{j}$, such that if $\left\{ \left(x_{j,k},I_{x_{j,k}}\right):1\leq k\leq n_{j}\right\} $
is a tagged partition of $\left[a_{j},b_{j}\right]$ subordinate to
$\delta_{j}$, then\[
\left|\sum_{x_{j,k}\in E}\Delta_{j,k}f\right|<\epsilon'/n.\]
Let $Q_{j}=\left\{ \left(x_{j,k},I_{x_{j,k}}\right):1\leq k\leq n_{j}\right\} $
be a tagged partition of $I_{j}$ subordinate to $\min\left\{ \delta,\delta_{j}\right\} $.

Also let $L$ be the subset of $\left\{ 1\ldots n\right\} $ such
that if $j\in L$ then $\Delta_{j}f\ge0$. Now let\[
R=\left\{ \left(x_{j},I_{j}\right):j\in L\right\} \cup\left(\bigcup_{j\notin L}Q_{j}\right).\]
Since $R$ is a tagged partition of $\left[a,b\right]$ subordinate
to $\delta$, \[
\epsilon>\left|\sum_{\substack{\left(x,\left[\alpha,\beta\right]\right)\in R\\
x\in E}
}f\left(\beta\right)-f\left(\alpha\right)\right|=\left|\sum_{\substack{j\in L\\
x_{j}\in E}
}\Delta_{j}f+\sum_{j\notin L}\sum_{x_{j,k}\in E}\Delta_{j,k}f\right|.\]
 Also, because each $Q_{j}$ is subordinate to $\delta_{j}$, \[
\epsilon'>\sum_{j\notin L}\left|\sum_{x_{j,k}\in E}\Delta_{j,k}f\right|\ge\left|\sum_{j\notin L}\sum_{x_{j,k}\in E}\Delta_{j,k}f\right|.\]
 Consequently, \[
\epsilon+\epsilon'>\left|\sum_{\substack{j\in L\\
x_{j}\in E}
}\Delta_{j}f\right|=\sum_{\substack{j\in L\\
x_{j}\in E}
}\left|\Delta_{j}f\right|.\]
 Since the choice of $\epsilon'>0$ was arbitrary, it must be true
that \[
\epsilon\ge\sum_{\substack{j\in L\\
x_{j}\in E}
}\left|\Delta_{j}f\right|.\]
 Similarly, \[
\epsilon\ge\sum_{\substack{j\notin L\\
x_{j}\in E}
}\left|\Delta_{j}f\right|.\]
 Therefore \[
2\epsilon\ge\sum_{x_{j}\in E}\left|\Delta_{j}f\right|,\]
 proving that $f$ has negligible variation on $E$.\end{proof}
\begin{cor}
\label{cor:lastcor}Assume that $g:\left[a,b\right]\rightarrow\mathbb{R}$
is differentiable almost everywhere and that $f:\mathbb{R}\rightarrow\mathbb{R}$
is HK integrable on every interval with endpoints in the range of
$g$. Define $F\left(x\right):=(HK)\int_{g\left(a\right)}^{x}f\left(u\right)du$.
Then $\left(f\circ g\right)\cdot g'$ is HK integrable on $\left[a,b\right]$
and the change of variables formula \[
(HK)\int_{g\left(\alpha\right)}^{g\left(\beta\right)}f\left(u\right)du=(HK)\int_{\alpha}^{\beta}f\left(g\left(s\right)\right)g'\left(s\right)ds\]
holds for every $\left[\alpha,\beta\right]\subseteq\left[a,b\right]$
if and only if $F\circ g$ has negligible variation on the set where
$\left(F\circ g\right)'=f\circ g\cdot g'$ fails.
\end{cor}
The necessary and sufficient condition that $F\circ g$ have negligible
variation on the set where $\left(F\circ g\right)'=f\circ g\cdot g'$
fails is clearly justified by the preceding lemma; in Theorem \ref{thm:equivcond}
of the next section, we prove that it is equivalent to the condition
that $F\circ g$ have negligible variation on each null set and on
the set where $g'=0$.
\begin{rem}
A tempting possibility to investigate is whether HK integrals automatically
satisfy the requirements of the substituting function in the change
of variables formula. In other words, if $g$ is an indefinite HK
integral, is it true that for any HK integrable $f$ that $(HK)\int_{g\left(\alpha\right)}^{g\left(\beta\right)}f\left(u\right)du=(HK)\int_{\alpha}^{\beta}f\left(g\left(s\right)\right)g'\left(s\right)ds$?
If this were true, then the composition of two indefinite HK integrals
would be an indefinite HK integral. However, this is not the case,
as the following two indefinite Lebesgue integrals (and hence also
HK integrals) fail to have a composition which is an HK integral.

Let $S$ be the Smith\textendash{}Volterra\textendash{}Cantor set
of measure $\frac{1}{2}$ \cite[pp. 90-91]{key-8} constructed on
a unit interval through the usual method of deleting an open interval
of length $4^{-n}$ from the center of each of the $2^{n-1}$ intervals
at step $n$, leaving a closed nowhere-dense set of positive measure
in the limit. Let $G\left(x\right)=\mbox{dist}\left(x,S\right)$ and
$F\left(x\right)=\sqrt[4]{x}$. Consequently $F$ and $G$ are both
indefinite Riemann integrals, yet for any $x\in S$ and $n\in\mathbb{N}$
there exists $y\in\left(x-2^{-n},x+2^{-n}\right)$ such that $\left(y-2^{-2n-3},y+2^{-2n-3}\right)\subseteq S^{c}$.
Thus $\left|\frac{F\left(G\left(y\right)\right)-F\left(G\left(x\right)\right)}{y-x}\right|=\frac{F\left(G\left(y\right)\right)}{\left|y-x\right|}>\frac{\sqrt[4]{2^{-2n-3}}}{2^{-n}}=\sqrt[4]{2^{2n-3}}$.
Therefore $F\circ G$ has no derivative on $S$, a set of positive
measure, and so cannot be the indefinite HK integral of any function.

Note that $G$ is differentiable except on the endpoints and midpoint
of each deleted interval, which is a countable set, and that the construction
of $G$ could be altered so that those points are differentiable as
well, with the same result.
\end{rem}

\section{\label{sec:Negligible-Variation}Negligible Variation}
\begin{example}
\label{exa:strictnegvar}The following is a continuous function that
has negligible conditional variation and not negligible variation
on a set.

Let $C$ denote the Cantor set and $c$ the Cantor-Lebesgue function
on $\left[0,1\right]$. Let $D=C\cup\left(-C\right)$. Let us denote
\[
\delta_{D}\left(x\right)=\begin{cases}
1 & \mbox{if }x\in D\\
\mbox{dist}\left(x,D\right) & \mbox{if }x\notin D\end{cases}.\]
If $\left\{ \left(x_{j},I_{j}\right):1\leq j\leq n\right\} $ is a
tagged partition of $\left[-1,1\right]$ subordinate to $\delta_{D}$,
then \begin{eqnarray*}
0 & = & \left|c\left(\left|1\right|\right)-c\left(\left|-1\right|\right)\right|\\
 & = & \left|\sum_{x_{j}\in D}\Delta_{j}c\left(\left|\cdot\right|\right)+\sum_{x_{j}\notin D}\overbrace{\Delta_{j}c\left(\left|\cdot\right|\right)}^{=0}\right|.\end{eqnarray*}

Thus $c\left(\left|\cdot\right|\right)$ has negligible conditional
variation on $D$.

Suppose there exists $\delta$ so that for $P\left[-1,1\right]=\left\{ \left(x_{j},I_{j}\right):1\leq j\leq n\right\} $
subordinate to $\delta$, $\sum_{x_{j}\in D}\left|\Delta_{j}c\left(\left|\cdot\right|\right)\right|<1$.
Form a tagged partition $P\left[-1,1\right]=\left\{ \left(x_{j},I_{j}\right):1\leq j\leq n\right\} $
as a union of a $Q_{1}\left[-1,0\right]$ and $Q_{2}\left[-1,0\right]$
subordinate to $\min\left\{ \delta_{D}\left(x\right),\delta_{1}\left(x\right)\right\} $.
These three partitions are a fortiori partitions subordinate to $\delta$.
Then, as above, \begin{eqnarray*}
2 & = & \left|c\left(\left|-1\right|\right)-c\left(\left|0\right|\right)\right|+\left|c\left(\left|0\right|\right)-c\left(\left|1\right|\right)\right|\\
 & = & \sum_{\substack{x_{j}\in D\\
x_{j}\in\left[-1,0\right]}
}\left|\Delta_{j}c\left(\left|\cdot\right|\right)\right|+\sum_{\substack{x_{j}\notin D\\
x_{j}\in\left[-1,0\right]}
}\overbrace{\left|\Delta_{j}c\left(\left|\cdot\right|\right)\right|}^{=0}\\
 &  & +\sum_{\substack{x_{j}\in D\\
x_{j}\in\left[0,1\right]}
}\left|\Delta_{j}c\left(\left|\cdot\right|\right)\right|+\sum_{\substack{x_{j}\notin D\\
x_{j}\in\left[0,1\right]}
}\overbrace{\left|\Delta_{j}c\left(\left|\cdot\right|\right)\right|}^{=0}.\end{eqnarray*}

Thus $c\left(\left|\cdot\right|\right)$ does not have negligible
variation on $D$. This argument also shows that $c\left(\left|\cdot\right|\right)$
does not have negligible conditional variation on $D\cap\left[0,1\right]$.
\end{example}

\subsection{\label{sub:Conditions-That-Imply}Conditions That Imply Negligible
Variation}
\begin{lem}
\label{lem:zeroderiv}Let $f:\left[a,b\right]\rightarrow\mathbb{R}$
be such that $f'\left(x\right)=0$ $\forall x\in D\subseteq\left[a,b\right]$.
Then $f$ has negligible variation on $D$.\end{lem}
\begin{proof}
Choose $\epsilon>0$. Let $\eta_{\epsilon}:D\rightarrow\mathbb{R}^{+}$
be a function such that if $y\in\left[x-\eta_{\epsilon}\left(x\right),x+\eta_{\epsilon}\left(x\right)\right]\cap\left[a,b\right]$
then \[
\left|f\left(y\right)-f\left(x\right)\right|\le\epsilon\left|y-x\right|/\left(b-a\right)\]
 and let \[
\delta\left(x\right)=\begin{cases}
\eta_{\epsilon}\left(x\right) & \mbox{if }x\in D\\
1 & \mbox{if }x\notin D\end{cases}\]

Choose a tagged partition $P\left[a,b\right]=\left\{ \left(x_{j},I_{j}\right):1\leq j\leq n\right\} $
subordinate to $\delta$. Then \[
\sum_{x_{j}\in D}\left|\Delta_{j}f\right|\le\left(\frac{\epsilon}{b-a}\right)\sum_{x_{j}\in D}\left|I_{j}\right|\le\epsilon.\]
\end{proof}
\begin{lem}
\label{lem:finderiv}Let $f:\left[a,b\right]\rightarrow\mathbb{R}$
have finite upper and lower Dini derivatives on a null set $Z$; that
is to say $\overline{D}f\left(x\right):=\limsup_{y\rightarrow x}\left|\frac{f\left(y\right)-f\left(x\right)}{y-x}\right|<\infty$
for all $x\in Z$. Then $f$ has negligible variation on $Z$.\end{lem}
\begin{proof}
Let $\epsilon>0$ and let $Z_{n}=Z\cap\left\{ x:\overline{D}f\left(x\right)\in\left[n,n+1\right)\right\} $.
Also, let $\eta_{1}:Z\rightarrow\mathbb{R}^{+}$ be a function such
that if $y\in\left[x-\eta_{1}\left(x\right),x+\eta_{1}\left(x\right)\right]\cap\left[a,b\right]$
then \[
\left|f\left(y\right)-f\left(x\right)\right|\le\left(1+\left\lfloor \overline{D}f\left(x\right)\right\rfloor \right)\left|y-x\right|\]

Let $C_{n}\supseteq Z_{n}$ be open sets with measure less than $\frac{\epsilon}{2^{n+1}\left(n+2\right)}$.
Define \[
\delta\left(x\right)=\begin{cases}
\min\left\{ \eta_{1}\left(x\right),\mbox{dist}\left(x,C_{n}^{c}\right)\right\}  & \mbox{if }x\in Z_{n}\\
1 & \mbox{if }x\notin Z\end{cases}\]

Choose a tagged partition $P\left[a,b\right]=\left\{ \left(x_{j},I_{j}\right):1\leq j\leq n\right\} $
subordinate to $\delta$. If $x_{j}\in Z_{n}$ for some $j$, then
$\left|\overline{D}f\left(x_{j}\right)\right|\in\left[n,n+1\right)$,
so \[
\left|\Delta_{j}f\right|\le\left|f\left(b_{j}\right)-f\left(x_{j}\right)\right|+\left|f\left(x_{j}\right)-f\left(a_{j}\right)\right|\le\left(n+2\right)\cdot\left|I_{j}\right|.\]
 Therefore \begin{eqnarray*}
\sum_{x_{j}\in Z}\left|\Delta_{j}f\right| & = & \sum_{n=0}^{\infty}\sum_{x_{j}\in Z_{n}}\left|\Delta_{j}f\right|\le\sum_{n=0}^{\infty}\left(n+2\right)\cdot\lambda\left(C_{n}\right)\le\epsilon,\end{eqnarray*}
 proving that $f$ has negligible variation on $Z$.
\end{proof}
Clearly, if a function has negligible variation on a set $N$, then
it has negligible conditional variation on $S\supseteq N$ if and
only if it has negligible conditional variation on $S\backslash N$.

Consider Theorem \ref{thm:subs1}, where $B$ is the set where $\left(F\circ g\right)'=f\circ g\cdot g'$
fails. Let $A_{F}$ and $A_{g}$ be the sets where $F$ and $g$ fail
to have have derivatives and set $A=g^{-1}\left(A_{F}\right)\cup A_{g}$.
By Lemma \ref{zeroset}, $g'\left(x\right)=0$ for almost every $x\in B$.
Similarly, $A_{F}$ and $A_{g}$ have measure zero and $g'$ is defined
almost everywhere, so $g'\left(x\right)=0$ at almost every $x\in A$.
Furthermore,\[
\left(F\circ g\right)'\left(x\right)=\begin{cases}
\left(F'\circ g\cdot g'\right)\left(x\right) & \mbox{if }x\in B\backslash A\\
\left(f\circ g\cdot g'\right)\left(x\right) & \mbox{if }x\in A\backslash B\end{cases}\]
so $\left(F\circ g\right)'$ is zero almost everywhere on $\left(B\backslash A\right)\cup\left(A\backslash B\right)$.
By Lemma \ref{lem:zeroderiv}, $F\circ g$ has negligible variation
on $\left(B\backslash A\right)\cup\left(A\backslash B\right)$. This
proves that if $F\circ g$ has negligible conditional variation on
any set $S$ such that $B\cap A\subseteq S\subseteq B\cup A$, then
it has it on every other such set.

%
{}
\begin{thm}
\label{thm:equivcond}Assume that $g:\left[a,b\right]\rightarrow D$
and $F:D\rightarrow\mathbb{R}$ have derivatives almost everywhere
and that $f=F'$ almost everywhere. Then $F\circ g$ has negligible
variation on the set where $\left(F\circ g\right)'=f\circ g\cdot g'$
fails if and only if $F\circ g$ has negligible variation on each
null set and on the set where $g'$ is zero.\end{thm}
\begin{proof}
Let $B$ again be the set where $\left(F\circ g\right)'=f\circ g\cdot g'$
fails and assume $F\circ g$ has negligible variation on $B$. It
then has negligible variation on all subsets of $B$ as well, including
$B\cap\left\{ x:g'\left(x\right)=0\right\} $. Since $\left(F\circ g\right)'=f\circ g\cdot g'$
on the complement of $B$, then, by Lemma \ref{lem:zeroderiv}, $F\circ g$
has negligible variation on $\left\{ x:g'\left(x\right)=0\right\} \backslash B$.
Therefore $F\circ g$ has negligible variation on the set $\left\{ x:g'\left(x\right)=0\right\} $.

Similarly, for any null set $Z$, $F\circ g$ will have negligible
variation on $Z\cap B$, since that is a subset of $B$. Also, by
Lemma \ref{lem:finderiv}, $F\circ g$ will have negligible variation
on $Z\backslash B$. Therefore $F\circ g$ has negligible variation
on $Z$.

Conversely, by Lemma \ref{zeroset}, there exists a null set $Z$
and a set $E\subseteq\left\{ x:g'\left(x\right)=0\right\} $ such
that $B=Z\cup E$. Consequently if $F\circ g$ has negligible variation
on each null set, it must have it on $Z$ in particular. Also, if
it has negligible variation on $\left\{ x:g'\left(x\right)=0\right\} $,
then it has it on its subsets such as $E$. Therefore $F\circ g$
has negligible variation on $B$.
\end{proof}
We may therefore restate Corollary \ref{cor:lastcor} in the following
equivalent form.
\begin{cor}
Assume that $g:\left[a,b\right]\rightarrow\mathbb{R}$ is differentiable
almost everywhere and that $f:\mathbb{R}\rightarrow\mathbb{R}$ is
HK integrable on every interval with endpoints in the range of $g$.
Define $F\left(x\right):=(HK)\int_{g\left(a\right)}^{x}f\left(u\right)du$.
Then $\left(f\circ g\right)\cdot g'$ is HK integrable on $\left[a,b\right]$
and the change of variables formula \[
(HK)\int_{g\left(\alpha\right)}^{g\left(\beta\right)}f\left(u\right)du=(HK)\int_{\alpha}^{\beta}f\left(g\left(s\right)\right)g'\left(s\right)ds\]
holds for every $\left[\alpha,\beta\right]\subseteq\left[a,b\right]$
if and only if $F\circ g$ has negligible variation on each null set
and on the set where $g'$ is zero.
\end{cor}

\subsection{\label{sub:Conditions-Implied-by}Implications of Negligible Variation}

Functions that satisfy the conditions of Theorem \ref{thm:svzero}
on a set $E$ have a derivative equal to zero almost everywhere on
$E$, so it follows from Lemma~\ref{lem:zeroderiv} that these functions
have negligible variation on a subset of $E$ of full measure. We
show next that the conclusion of Theorem~\ref{thm:svzero} holds
for this larger class of functions.
\begin{thm}
\label{thm:negvarderiv}If $g:\left[a,b\right]\rightarrow\mathbb{R}$
has negligible variation on $E\subseteq\left[a,b\right]$, then $\overline{D}g\left(x\right)=0$
almost everywhere on $E$.\end{thm}
\begin{proof}
Let $\overline{D}_{+}g\left(x\right):=\limsup_{h\rightarrow0+}\left|\frac{g\left(x+h\right)-g\left(x\right)}{h}\right|$
and $E_{\gamma}=\left\{ x\in E\backslash\left\{ b\right\} :\overline{D}_{+}g\left(x\right)>\gamma\right\} $.
For reasons of symmetry, it suffices to show that $\lambda\left(E_{0}\right)=0$.
Furthermore, $\lambda^{*}\left(E_{0}\right)\le\sum_{n=1}^{\infty}\lambda^{*}\left(E_{\frac{1}{n}}\right)$,
so it is sufficient to show $\lambda^{*}\left(E_{\gamma}\right)=0$
for every $\gamma>0$.

Choose $\gamma,\epsilon>0$. Let $\delta$ be a gauge such that for
any tagged partition $P\left[a,b\right]=\left\{ \left(x_{j},I_{j}\right):1\leq j\leq n\right\} $
subordinate to $\delta$, $\sum_{x_{j}\in E}\left|\Delta_{j}g\right|<\epsilon\gamma/4$.
Let \[
C=\left\{ \left[\alpha,\beta\right]\subseteq\left[a,b\right]:0<\beta-\alpha<\delta\left(\alpha\right)\mbox{ and }\alpha\in E_{\gamma}\mbox{ and }\left|\frac{g\left(\beta\right)-g\left(\alpha\right)}{\beta-\alpha}\right|>\gamma/2\right\} .\]
 Since $C$ is a Vitali cover of $E_{\gamma}$, there is a finite
collection $D$ of disjoint intervals in $C$ such that $\lambda^{*}\left(E_{\gamma}\backslash\bigcup D\right)<\epsilon/2$.

Because $D$ is a finite collection of closed intervals, there exists
a finite collection $O$ of disjoint open (in $\left[a,b\right]$)
intervals complementing $D$. By Cousin's Lemma, for each $\left(\alpha,\beta\right)\in O$
there exists a tagged partition $Q\left[\alpha,\beta\right]$ subordinate
to $\delta$. Let \[
P\left[a,b\right]=\left\{ \left(\alpha,\left[\alpha,\beta\right]\right):\left[\alpha,\beta\right]\in D\right\} \cup\left(\bigcup_{\left(\alpha,\beta\right)\in O}Q\left[\alpha,\beta\right]\right).\]
 Hence $P\left[a,b\right]$ is a tagged partition $\left\{ \left(x_{j},I_{j}\right):1\leq j\leq n\right\} $
subordinate to $\delta$. So \[
\epsilon\gamma/4>\sum_{x_{j}\in E}\left|\Delta_{j}g\right|\ge\sum_{I_{j}\in D}\left|\Delta_{j}g\right|\ge\sum_{I_{j}\in D}\left|I_{j}\right|\cdot\gamma/2.\]
 Therefore $\lambda\left(\bigcup D\right)=\sum_{I_{j}\in D}\left|I_{j}\right|<\epsilon/2$
and so $\lambda^{*}\left(E_{\gamma}\right)\le\lambda\left(\bigcup D\right)+\lambda^{*}\left(E_{\gamma}\backslash\bigcup D\right)<\epsilon$.
Since $\epsilon$ was arbitrary, $0=\lambda\left(E_{\gamma}\right)$.
\end{proof}
Theorem \ref{thm:negvarderiv} and Lemma \ref{lem:zeroderiv} show
that $f:\left[a,b\right]\rightarrow\mathbb{R}$ has negligible variation
on a set $E\subseteq\left[a,b\right]$ if and only if there exists
a null set $Z\subseteq E$ such that $f$ has negligible variation
on $Z$ and $f'\left(x\right)=0$ for all $x\in E\backslash Z$.
\begin{thm}
If $g:\left[a,b\right]\rightarrow\mathbb{R}$ has negligible variation
on $E\subseteq\left[a,b\right]$, then $\lambda\left(g\left(E\right)\right)=0$.\end{thm}
\begin{proof}
We can clearly assume that $a,b\notin E$.

Choose $\epsilon>0$ and let $\delta$ be a gauge such that for any
tagged partition $P\left[a,b\right]=\left\{ \left(x_{j},I_{j}\right):1\leq j\leq m\right\} $
subordinate to $\delta$, $\sum_{x_{j}\in E}\left|\Delta_{j}g\right|<\epsilon$.
Let $\eta_{x}=\min\left\{ b-x,x-a,\delta\left(x\right)\right\} $.
Thus for $x\in E$ \[
\sup_{\left|h\right|\le\eta_{x}}\left|g\left(x+h\right)-g\left(x\right)\right|\le\epsilon,\]
 and we denote this finite-valued supremum as $s\left(x\right)$.
For every $x\in E$, choose $h_{x}$ so that $\left|g\left(x+h_{x}\right)-g\left(x\right)\right|\ge s\left(x\right)/2$
and $\left|h_{x}\right|\le\eta_{x}$.

Define $C\left(T\right)=\left\{ \left(x-\left|h_{x}\right|,x+\left|h_{x}\right|\right):x\in T\right\} $.
$C\left(E\right)$ is a Besicovitch cover of $E$, so there exist
two sequences%
\footnote{While the statement of Besicovitch's Covering Theorem is usually given
in $\mathbb{R}^{d}$ and with only rough bounds for the number of
sequences necessary, it is not too hard to show that two sequences
suffice for $\mathbb{R}^{1}$.%
} (possibly finite) of distinct points from $E$, $\left\{ y_{i}\right\} $
and $\left\{ z_{i}\right\} $, such that $C\left(\left\{ y_{i}\right\} \right),C\left(\left\{ z_{i}\right\} \right)$
each consist of disjoint intervals and $C\left(\left\{ y_{i}\right\} \right)\cup C\left(\left\{ z_{i}\right\} \right)$
covers $E$.

Since the closure of $C\left(\left\{ y_{i}\right\} _{i=1}^{n}\right)$
is a finite union of closed intervals, there exists a finite collection
$O_{n}$ of disjoint open (in $\left[a,b\right]$) intervals complementing
it. Also, for each $\left(\alpha,\beta\right)\subseteq\left[a,b\right]$
there exists $Q\left[\alpha,\beta\right]$ subordinate to $\delta$.
Let \begin{eqnarray*}
P_{n}\left[a,b\right] & = & \left(\bigcup_{i=1}^{n}\left\{ \left(y_{i},\left[y_{i}-\left|h_{x}\right|,y_{i}\right]\right),\left(y_{i},\left[y_{i},y_{i}+\left|h_{x}\right|\right]\right)\right\} \right)\\
 &  & \cup\left(\bigcup_{\left(\alpha,\beta\right)\in O_{n}}Q\left[\alpha,\beta\right]\right).\end{eqnarray*}
 Hence $P_{n}\left[a,b\right]$ is a tagged partition subordinate
to $\delta$ and so \[
\epsilon>\sum_{i=1}^{n}\left|g\left(y_{i}+h_{x}\right)-g\left(y_{i}\right)\right|\ge\sum_{i=1}^{n}s\left(y_{i}\right)/2.\]
 Define $D_{n}=\bigcup_{i=1}^{n}\left[g\left(y_{i}\right)-s\left(y_{i}\right),g\left(y_{i}\right)+s\left(y_{i}\right)\right]$.
Then $D_{n+1}\supseteq D_{n}$ and, from the inequality above, $4\epsilon>\lambda\left(D_{n}\right)$
for all $n$. Additionally, $\bigcup_{n=1}^{\infty}D_{n}\supseteq g\left(C\left(\left\{ y_{i}\right\} \right)\right)$.
Thus $4\epsilon\ge\lambda\left(\bigcup_{n=1}^{\infty}D_{n}\right)\ge\lambda^{*}\left(g\left(C\left(\left\{ y_{i}\right\} \right)\right)\right)$.
The same argument applies for $\left\{ z_{i}\right\} $, so \[
8\epsilon\ge\lambda^{*}\left(g\left(C\left(\left\{ y_{i}\right\} \right)\cup C\left(\left\{ z_{i}\right\} \right)\right)\right)\ge\lambda^{*}\left(g\left(E\right)\right).\]
 Since $\epsilon$ was arbitrarily small, $\lambda\left(g\left(E\right)\right)=0$.
\end{proof}


\begin{bibdiv}
\begin{biblist}
\bib{key-1}{book}{    author={Bartle, Robert G.},    title={A modern theory of integration},    series={Graduate Studies in Mathematics},    volume={32},    publisher={American Mathematical Society},    place={Providence, RI},    date={2001},    pages={xiv+458}, label={Bartle},    isbn={0-8218-0845-1},    review={\MR{1817647 (2002d:26001)}}, }
\bib{key-8}{book}{    author={Bressoud, David M.},    title={A radical approach to Lebesgue's theory of integration},    series={MAA Textbooks},    publisher={Cambridge University Press},    place={Cambridge},    date={2008},    pages={xiv+329},    isbn={978-0-521-71183-8},    isbn={0-521-71183-5},    review={\MR{2380238 (2008j:00001)}}, label={Bressoud},} 
\bib{key-7}{article}{    author={Goodman, Gerald S.},    title={Integration by substitution},    journal={Proc. Amer. Math. Soc.},    volume={70},    date={1978},    number={1},    pages={89--91},    issn={0002-9939},    review={\MR{0476952 (57 \#16497)}}, label={Goodman},}
\bib{key-6}{book}{    author={Gordon, Russell A.},    title={The integrals of Lebesgue, Denjoy, Perron, and Henstock},    series={Graduate Studies in Mathematics},    volume={4},    publisher={American Mathematical Society},    place={Providence, RI},    date={1994},    pages={xii+395},    isbn={0-8218-3805-9},    review={\MR{1288751 (95m:26010)}}, label={Gordon}}
\bib{key-5}{article}{    author={Krzy{\.z}ewski, K.},    title={On change of variable in the Denjoy-Perron integral. I},    journal={Colloq. Math.},    volume={9},    date={1962},    pages={99--104},    issn={0010-1354},    review={\MR{0132816 (24 \#A2652)}}, label={Krz1}}
\bib{Krz2}{article}{    author={Krzy{\.z}ewski, K.},    title={On change of variable in the Denjoy-Perron integral. II},    journal={Colloq. Math.},    volume={9},    date={1962},    pages={317--323},    issn={0010-1354},    review={\MR{0142714 (26 \#283)}}, label={Krz2}}
\bib{key-4}{book}{    author={Lee, Peng Yee},    author={V{\'y}born{\'y}, Rudolf},    title={Integral: an easy approach after Kurzweil and Henstock},    series={Australian Mathematical Society Lecture Series},    volume={14},    publisher={Cambridge University Press},    place={Cambridge},    date={2000},    pages={xii+311},    isbn={0-521-77968-5},    review={\MR{1756319 (2001h:26011)}}, label={Lee \& Vyborny}}
\bib{key-2}{article}{    author={Serrin, James},    author={Varberg, Dale E.},    title={A general chain rule for derivatives and the change of variables    formula for the Lebesgue integral},    journal={Amer. Math. Monthly},    volume={76},    date={1969},    pages={514--520},    issn={0002-9890},    review={\MR{0247011 (40 \#280)}}, label={S\&V}}
\bib{key-9}{article}{    author={V{\'y}born{\'y}, Rudolf},    title={Some applications of Kurzweil-Henstock integration},    journal={Math. Bohem.},    volume={118},    date={1993},    number={4},    pages={425--441},    issn={0862-7959},    review={\MR{1251885 (94k:26014)}}, label={Vyborny}}
\end{biblist}
\end{bibdiv}
\end{document}